\documentclass{article}

\usepackage{microtype}
\usepackage{graphicx}
\usepackage{booktabs} 
\usepackage{amsmath,amssymb,amsthm,textcomp}
\usepackage{enumitem}
\usepackage{caption}
\usepackage{subcaption}
\usepackage{tabularx}
\usepackage{multirow}
\usepackage[table,xcdraw]{xcolor}
\usepackage{epstopdf}
\usepackage{array}
\usepackage{mathtools}
\usepackage[most]{tcolorbox}
\newcolumntype{L}[1]{>{\raggedright\let\newline\\\arraybackslash\hspace{0pt}}m{#1}}
\newcolumntype{C}[1]{>{\centering\let\newline\\\arraybackslash\hspace{0pt}}m{#1}}
\newcolumntype{R}[1]{>{\raggedleft\let\newline\\\arraybackslash\hspace{0pt}}m{#1}}

\usepackage{hyperref}

\newtheorem{theorem}{Theorem}[section]

\newtheorem{lemma}[theorem]{Lemma}

\newtheorem{remark}[theorem]{Remark}

\newtheorem{proposition}[theorem]{Proposition}

\newcommand{\be}{\begin{equation}}
\newcommand{\ee}{\end{equation}}

\colorlet{Mycolor1}{green!10!orange!90!}

\newcommand{\armin}[1]{{\color{black} #1}}

\newcommand{\R}{\mathbb{R}}
\newcommand{\St}{\mathrm{s.t.}}
\newcommand{\mc}{\mathcal}
\newcommand{\opt}{^\star}
\newcommand{\ds}{\displaystyle}
\newcommand{\diag}{\mathrm{diag}}
\newcommand{\BR}{\mathrm{BR}}
\newcommand{\Let}{\triangleq}

\newcommand{\Sym}{\mathbb{S}}
\newcommand{\PSD}{\Sym_+}

\newcommand{\J}{J^\star}
\renewcommand{\d}{d^\star} 
\newcommand{\ind}{\mathds{1}}

\usepackage{float,dsfont}
\usepackage[accepted]{icml2021}

\icmltitlerunning{Principal Component Hierarchy for Sparse Quadratic Programs}

\begin{document}

\twocolumn[
\icmltitle{Principal Component Hierarchy for Sparse Quadratic Programs}




\begin{icmlauthorlist}
\icmlauthor{Robbie Vreugdenhil}{delft}
\icmlauthor{Viet Anh Nguyen}{stanford,vinai}
\icmlauthor{Armin Eftekhari}{umea}
\icmlauthor{Peyman Mohajerin Esfahani}{delft}
\end{icmlauthorlist}

\icmlaffiliation{delft}{Delft Center for Systems and Control, Delft University of Technology}
\icmlaffiliation{stanford}{Department of Management Science and Engineering, Stanford University}
\icmlaffiliation{vinai}{VinAI Research, Vietnam}
\icmlaffiliation{umea}{Department of Mathematics and Mathematical Statistics, Umea University}

\icmlcorrespondingauthor{Robbie Vreugdenhil}{robbievreugdenhil@gmail.com}

\icmlkeywords{Machine Learning, ICML}

\vskip 0.3in
]



\printAffiliationsAndNotice{}  


\begin{abstract}
We propose a novel approximation hierarchy for cardinality-constrained, convex quadratic programs that exploits the rank-dominating eigenvectors of the quadratic matrix. Each level of approximation admits a min-max characterization whose objective function can be optimized over the binary variables analytically, while preserving convexity in the continuous variables. Exploiting this property, we propose two scalable optimization algorithms, coined as the ``{\em best response}" and the ``{\em dual program}", that can efficiently screen the potential indices of the nonzero elements of the original program. We show that the proposed methods are competitive with the existing screening methods in the current sparse regression literature, and it is particularly fast on instances with high number of measurements in experiments with both synthetic and real datasets.
\end{abstract}

\section{Introduction} \label{sec:introduction}

Sparsity is a powerful inductive bias that improves the interpretability and performance of many regression models~\cite{ribeiro2016should,ref:hastie2009elements,ref:bertsimas2017trimmed}. Recent years have witnessed a growing interest in sparsity-based methods and algorithms for sparse recovery, mostly in the setting of sparse linear regression~\cite{ref:atamturk2020safe,ref:bertsimas2017sparse,ref:hazimeh2020sparse,ref:hastie2017extended}. 

In this paper we study \armin{the} more general problem of sparse linearly-constrained quadratic programming with a regularization term. Sparsity is imposed in this context by controlling the $\ell_0$ norm of the estimator \cite{ref:miller2002subset}. More specifically, we consider the problem 
\begin{equation} \label{eq:P} \tag{$\mathcal P$}
\begin{array}{rcl}
    \J \Let & \min & {\langle c, x\rangle}+\langle x, Q x\rangle + {\eta^{-1}}\|x\|_2^2 \vspace{1mm}\\
    & \St & x\in\R^n \vspace{1mm}\\
    & & A x \leq b, \quad \|x\|_0 \leq s,
\end{array}
\end{equation}
 where $Q \in \PSD^n$, $c \in \R^n$, $A \in \R^{m \times n}$, $b \in \R^m$, \armin{and the integer~$s\le n$ specifies the target sparsity level.}  
 The ridge regularization term $\eta^{-1} \|x\|_2^2$~\cite{ref:hoerl1970ridge} in the objective function reduces the mean squared error when the data is affected by noise and/or uncertainty \cite{ref:ghaoui1997robust,ref:mazumder2020subset}. 
 
 
Problem~\eqref{eq:P} arises in a wide 
\armin{range}
of applications,
\armin{including}
sparse linear regression, model predictive control \cite{ref:aguilera2014quadratic}, portfolio optimization \cite{ref:bertsimas2020scalable}, binary quadratic programming and principal component analysis  \cite{ref:bertsimas2020unified}.

\textbf{Motivation for our approach.}
Throughout, suppose that~$Q$ in problem~\eqref{eq:P} admits the eigendecomposition~$Q=\sum_{i=1}^{n} \lambda_{i}  v_{i} v_{i}^\top$, 
where $\lambda_1 \ge \cdots \ge \lambda_n \ge 0$ are the eigenvalues of $Q$. Our approach to solve~\eqref{eq:P} is centered around the following key observation: 
\begin{tcolorbox}[colback=white!5!white,colframe=black!75!black]
  \textbf{Observation.} 
  For many real-world applications, the matrix $Q$ is \textit{nearly low-rank}.
\end{tcolorbox}
The concept of \textit{nearly low-rank} is contextual. In this paper, we say that $Q$ is nearly low-rank if 
a low-rank matrix can approximate $Q$ to a reasonable accuracy, where the accuracy is measured by a matrix norm.



It is instructive to demonstrate the nearly low-rank property of $Q$ in the context of linear regression. Given~$N$ training samples $\{(\xi_i, \omega_i)\}_{i=1}^N \subset \R^{n}\times \R$, the sparse ridge regression problem
 \begin{equation}
    \begin{array}{cl}
        \min\limits_{x: \|x\|_0 \le s} & \displaystyle \frac{1}{N} \sum_{i=1}^N \|\xi_i - x^\top \omega_i\|_2^2 + \eta^{-1} \|x\|_2^2
    \end{array}
    \label{eq:ridge}
 \end{equation}
 coincides with~\eqref{eq:P} for the choice of 
 \be \label{eq:Qc}
 Q = \frac{1}{N} \sum_{i=1}^{N} \omega_i \omega_i^\top, \quad  c = \frac{1}{N} \sum_{i=1}^N \xi_i \omega_i.
 \ee
 In high-dimensional regression ($N<n$), the matrix $Q$ is automatically rank-deficient. Moreover, $Q$ 
 can also be nearly low-rank
 even when $N > n$.
\armin{As a numerical example,} for the UCI Superconductivity dataset\footnote{Available at \url{https://archive.ics.uci.edu/ml/datasets/Superconductivty+Data}}, we randomly select~$70\%$ of the dataset as training samples, and
\armin{then} compute the matrix~$Q$, as specified in~\eqref{eq:Qc}. Figure~\ref{fig:EVdecay} plots the empirical distribution of the largest eigenvalues of $Q$, taken over 100 independent replications. On average, the ratio~$\lambda_1/\lambda_{10}$ between the $1^{\text{st}}$ largest and the $10^{\text{th}}$ largest eigenvalues of $Q$ is 97.2. We can also observe that the magnitude of the eigenvalues decays 
relatively fast for this dataset. 

\begin{figure}[H]
    \centering
    \includegraphics[width=7cm]{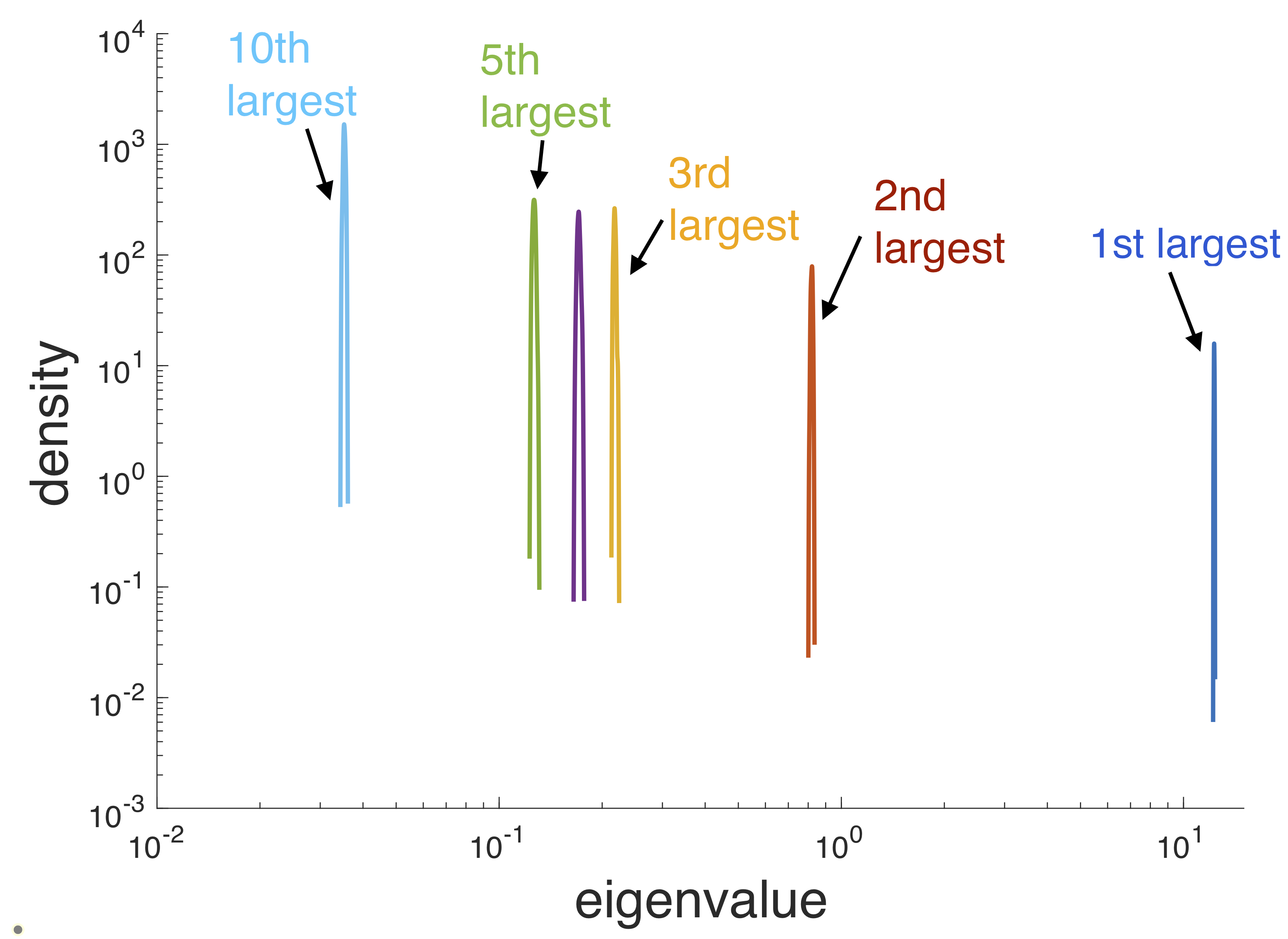}
    \caption{Empirical distribution of the eigenvalues of the matrix $Q$ for the Superconductivity dataset (plotted in log-log scale).}
   \label{fig:EVdecay}
\end{figure}
Next, let us recall the low-rank truncation of the matrix~$Q$. For $k\le n$,
the approximation of $Q$ using its $k$ leading eigenvectors, denoted by~$Q_{k}$, is given by
\begin{align*}
Q_{k} &\triangleq \sum_{i=1} ^{k} \lambda_{i}  v_{i} v_{i} ^\top.
\end{align*}
Following the same sampling procedure described above, Figure~\ref{fig:Frobdecay} depicts how the average truncation error, measured by the Frobenius norm $\| Q_k - Q\|_F$, rapidly decreases as $k$ increases. \armin{In this figure,} we also report the minimum 
\armin{dimension}~$\hat{k}$ so that the reconstruction error of $Q_{\hat{k}}$ falls below $10\%$ that of the rank-1 approximation $Q_1$. We observe that~$\hat{k} \ll n$ for many UCI regression datasets. \armin{This nearly low-rank behavior is typical in data sciences~\cite{udell2019big}.}
\begin{figure}[h!]
    \centering
    \includegraphics[width=8cm, trim={0 9cm 0  9cm},clip]{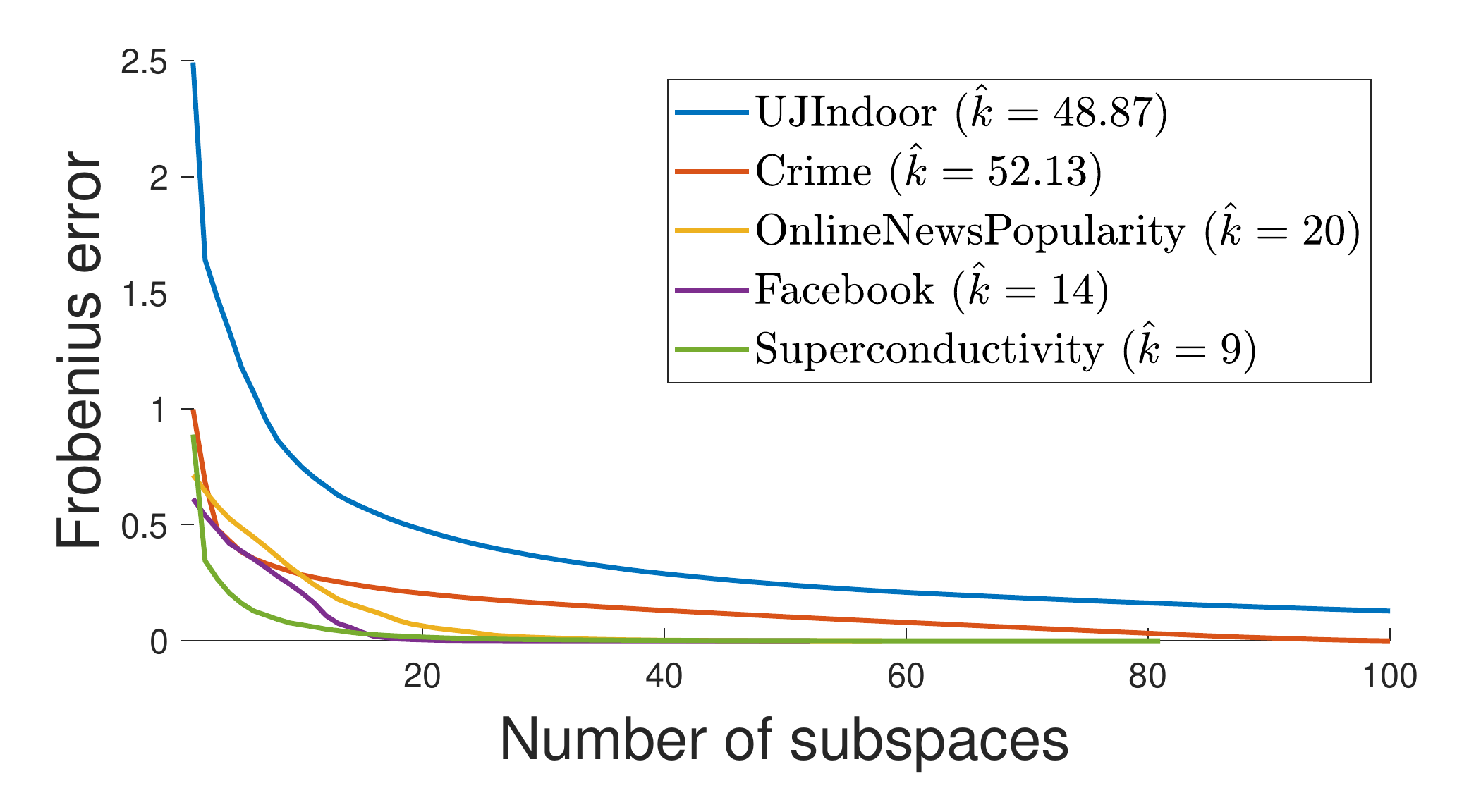}
    \caption{Empirical Frobenius reconstruction error~$\|Q_k - Q\|_F$ averaged over 100 independent replications.}
    \label{fig:Frobdecay}
\end{figure}

\textbf{Contributions.}~We summarize the contributions as follows.

\begin{enumerate}[label=(\roman*), leftmargin=7mm]
    \item {\bf Hierarchy \armin{of} approximations and min-max characterization:}
    We exploit the nearly low-rank nature of the 
    matrix $Q$ and propose a hierarchy of approximations for cardinality-constrained convex quadratic programs. 
    This hierarchy enables us to strike a balance between the scalability of the 
    \armin{solver}
    and the quality of the solution. We further show that each $k$-leading \armin{spectral truncation}
    \armin{can be characterized as}
    a min-max problem~(Proposition~\ref{prop:minmax}).

    \item {\bf Scalable and certifiable algorithms:} We propose two scalable algorithms that enjoy optimality certificates if the min-max characterization admits a saddle point (Propositions~\ref{prop:BR} and \ref{prop:DA}). The proposed algorithms build on a desirable feature of the objective function of the min-max characterization, in which the minimizer over the binary variables admits a closed-form solution~(Lemma~\ref{lem:closed form}). 
    
    \item {\bf Safe screening:} If the min-max characterizations do not admit a saddle point, the proposed iterative algorithms can serve as a screening method to reduce the variables in the original problem. We investigate the effectiveness of our algorithms through an in-depth numerical comparison with the recent sparse regression literature including the safe screening procedures of \citet{ref:atamturk2020safe} and the warm start of \citet{ref:bertsimas2017sparse}. Moreover, we also benchmark against direct optimization methods~\citet{ref:beck2012sparsity} and \citet{ref:ganzhao2020block}. Experiments on both synthetic and real datasets reveal that our algorithms deliver promising results (Section~\ref{sec:numerical}). 
    \end{enumerate}

We note that using leading principal components in regression has a long-standing 
\armin{history}~\cite{ref:naes1988principal,ref:hastie2009elements,ref:baye1984combining}. However, to the best of our knowledge, it has never been applied 
\armin{in the context of}
cardinality-constrained convex quadratic problems. Compared to the existing solution procedures in the literature, the performance of our methods, measured by both the objective value and the screening capacity, is consistent across a wider range of input parameters. Moreover, \armin{in sparse regression,} our methods can scale to problems of large sample size $N$ because the complexity of our methods does not  depend on $N$ \armin{poorly.}



\begin{figure*}[h!]
    \begin{center}
    \includegraphics[width=0.8\linewidth]{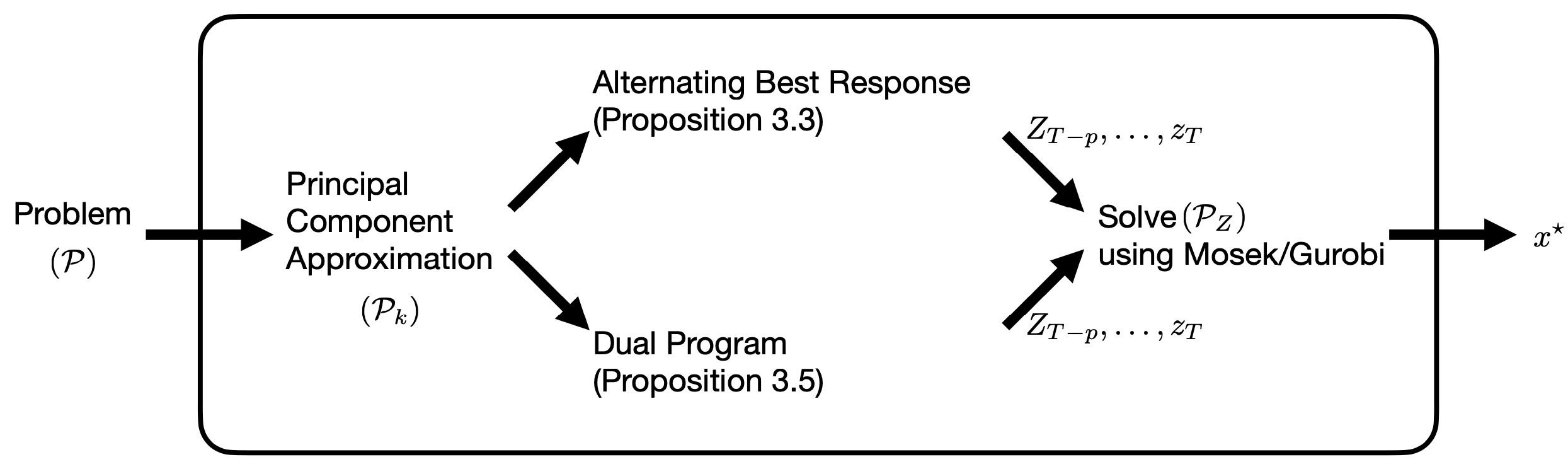}
    \vspace{-2mm}
    \caption{Schematic overview of the principal component approximation to sparsity-constrained quadratic programs.}
    \label{fig:pipeline}
    \end{center}
    \vspace{-2mm}
\end{figure*}

This paper unfolds as follows. Section~\ref{sec:literature} provides a brief overview over the landscape of the cardinality-constrained quadratic problems. Section~\ref{sec:PCA} devises two distinctive algorithms that leverage the principal component approximation of the matrix $Q$. Section~\ref{sec:numerical} reports an in-depth numerical comparison between our algorithms and the current sparse regression literature.

\textbf{Notations.} For any matrix $u$, we use $\sqrt{u}$ and $|u|$ to denote the component-wise square root and absolute value of $u$, respectively. For a vector~$u$, we use $\diag(u)$ to denote the diagonal matrix formed by $u$. For an integer $n$, we also denote with $[n]$ the set of integers $\{1,\dots,n\}$. Given an index set~$\mathbb A \subset [n]$, the binary vector~$u = \ind_n({\mathbb A}) \in \{0,1\}^n$ is defined as $u_j = 1$ if and only if $j \in \mathbb A$. In words, $\ind_n(\mathbb{A})$ is the indicator vector of the index set $\mathbb{A}$.


\section{Literature Review} 
\label{sec:literature}

\textbf{Sparse regression. }
In general, the sparse regression problem is NP-hard \cite{ref:natarajan1995sparse}. 
Until recently the sparse optimization literature largely focused on convex heuristics for sparse regression, e.g., Lasso ($\ell_1$) \cite{ref:tibshirani1996regression} and Elastic Net ($\ell_1 -\ell_2$) \cite{ref:zou2005regularization}.
Despite their scalability, convex heuristic approaches are inherently biased since the $\ell_1$-norm penalizes large and small coefficients uniformly. In contrast, solvers that directly tackle sparse regression do not suffer from unwanted shrinkage and have enjoyed a resurgence of interest~\cite{ref:bertsimas2017sparse,ref:hazimeh2020sparse,ref:dedieu2020learning,ref:gomez2018mixed}, thanks to the recent
breakthroughs in mixed-integer programming \cite{ref:bertsimas2016best}. These direct approaches are also the focus of this paper.

In particular, \citet{ref:bertsimas2017sparse} devise a convex reformulation of the sparse regression problem using duality theory, the solution of which provides a warm start for a branch-and-cut algorithm. This method can solve sparse regression problems at the scale of $n \approx 10^5$ while the earlier work~\cite{ref:bertsimas2016best} only goes to sizes of $n \approx 10^3$. However, as pointed out by \citet{ref:xie2020scalable}, the performance of these algorithms depends critically on the speed of the commercial solvers and varies significantly from one dataset to another. Therefore we focus on the warm start method in the approach of~\citet{ref:bertsimas2017sparse}, which makes use of the kernel matrix~$[\omega_i^\top \omega_j]_{i,j}$. Notice that the size of this kernel matrix scales with the number of samples. On the contrary, our proposed solution procedures use only the resulting matrix $Q$, whose size does not depend on the number of samples. 

Convex approximations of the sparse regression can however be used as a safe screening procedure as demonstrated by \citet{ref:atamturk2020safe}. Safe screening methods aim to identify the support of the solution set of~\eqref{eq:P} and reduce the problem dimension $n$ before invoking an MIQP solver.
Indeed, if we correctly rule out any single suboptimal dimension, the solution space would be cut by half. In this sense, the expected speedup for the MIQP solver is exponential \cite{ref:atamturk2020safe}.

\textbf{Sparse quadratic programs. }
We are only aware of a few papers that solve sparse quadratic programs exactly. 
\citet{ref:beck2012sparsity} and~\citet{ref:beck2015on} devise coordinate descent type algorithms based on the concept of coordinate-wise optimality, which updates the support at each iteration and, in particular, can also be used to solve sparse quadratic programs. \citet{ref:ganzhao2020block} considers a block decomposition algorithm that combines combinatorial search and coordinate descent. Specifically, this method uses a random or greedy strategy to find the working set and then performs a global combinatorial search over the working set based on the original objective function. Recently, \citet{ref:bertsimas2020scalable} and \citet{ref:bertsimas2020unified} apply the advances in exact sparse regression to sparse quadratic programs to solve problems of higher dimension. Both essentially rewrite \eqref{eq:P} as a sparse regression problem
and then solve it with a modified version of the branch-and-cut algorithm of \citet{ref:bertsimas2017sparse}. This procedure can also accommodate
linear constraints. 


Given the existing literature mentioned above, to our best of knowledge, our approach is the first to
identify and exploit the low-rank structure of $Q$ in the original problem~\eqref{eq:P} by using the leading principal component approximation of $Q$, notably in the context of sparse quadratic programming.
\section{Principal Component Approximation}
\label{sec:PCA}

The key idea of this study is to leverage the principal component approximation of the matrix $Q$ in \eqref{eq:P} in order to deploy the duality technique from convex optimization in a more efficient manner. To this end, we introduce additional continuous variables~$y$ along with the equality constraints
$ \sqrt{\lambda_i} y_{i}= \sqrt{\lambda_i} \left\langle v_{i}, x\right\rangle$ where $\lambda_i$ and $v_i$ are, respectively, the eigenvalues and eigenvectors of the matrix~$Q$. Throughout, we denote $V = [v_1, \ldots, v_k] \in \R^{d \times k}$. Using these definitions, program~\eqref{eq:P} can be approximated via
\begin{equation} \label{eq:P_k} \tag{$\mathcal{P}_k$}
\begin{array}{rcl}
\J_k \Let & \min & \langle c, x\rangle + \sum\limits_{i=1} ^{k} \lambda_{i}  y_{i}^{2} + {\eta^{-1}} \|x\|_2^2 \vspace{1mm}\\
    & \St & x \in\mathbb{R}^n,~y \in \R^k \vspace{1mm}\\ 
    && A x \leq b, \quad \|x\|_0 \leq s \vspace{1mm}\\
    & & \sqrt{\lambda_{i}}y_{i} = \sqrt{\lambda_{i}}\left\langle v_{i}, x\right\rangle, \quad  i\in [k].
\end{array}
\end{equation}
The last equality constraints of~\eqref{eq:P_k} are scaled using the coefficients~$\sqrt{\lambda_i}$ to improve numerical stability. Since the matrix~$Q$ is positive semidefinite, we have the hierarchy of approximations in which the sequence of optimal values~$\J_k$ preserves the order
\[
    \J_1 \le \J_2 \le \cdots \le \J_n = \J,
\]
where $\J$ is the optimal value of program~\eqref{eq:P}. In fact, one can observe that the two programs~\eqref{eq:P} and \eqref{eq:P_k} are equivalent when $k = n$. Next, we use the standard convex duality to turn program~\eqref{eq:P_k} into a min-max optimization problem. 
\begin{proposition}[Min-max characterization] \label{prop:minmax}
For each $k \le n$, the optimal value~$\J_k$ of \eqref{eq:P_k} is equal to 
\be\tag{$\mc M_k$} \label{eq:M_k}
    \J_k = \min\limits_{\substack{z \in \{0, 1\}^n \\ \sum z_j \le s}}~
    \max\limits_{\substack{ \alpha \in \mathbb{R} ^{k} \\\beta \in \R_+^m}}~L(z, \alpha, \beta),
\ee
where the objective function $L$ is defined as
\begin{align}\label{eq:L}
&L(z, \alpha, \beta) \Let -\beta^\top b -\frac{1}{4} \| \alpha\|_2^2  \\
& -\frac{\eta}{4}  (c+ V \sqrt{\Lambda} \alpha + A^\top \beta)^\top \diag(z) (c+ V \sqrt{\Lambda} \alpha + A^\top \beta), \notag
\end{align}
in which $\Lambda = \diag\{\lambda_1, \cdots,\lambda_k\}$ is a diagonal matrix whose elements on the main diagonal are the first $k$ largest eigenvalues of the matrix~$Q$. Moreover, the nonzero coordinates of the optimal variable~$x^\star$ in~\eqref{eq:P_k} contain the nonzero elements of the optimal solution~$z^\star$ in~\eqref{eq:M_k}.
\end{proposition}
Thanks to Proposition~\ref{prop:minmax}, the information regarding the support of the optimal solution~$x^\star$ in~\eqref{eq:P_k} can effectively be obtained from the support of the optimal solution~$z^\star$ in~\eqref{eq:M_k}. We note that 
finding the support is indeed the computational bottleneck of program~\eqref{eq:P_k}. A key feature of the objective function~$L$ of program~\eqref{eq:M_k} is that when the variables~$(\alpha,\beta)$ is fixed, the minimization over the binary variable~$z$ can be solved analytically. 

\begin{lemma}[Closed-form minimizer]\label{lem:closed form}
Given any pair~$(\alpha, \beta)$, the minimizer of the function~$L$ defined in~\eqref{eq:L} can be computed as
\begin{align}\label{eq:J-set}
    &\arg\min_{\substack{z \in \{0, 1\}^n \\ \sum z_j \le s}} L(z, \alpha, \beta)  = \ind_n\big(\mathcal{J}(\alpha,\beta)\big), \quad \text{where}\\
    & \mathcal{J}(\alpha,\beta) \Let \left\{ j \in [n] :\begin{array}{c}
         \text{$j$ is an index of the $s$-largest}  \\
         \text{elements of the vector} \\
         ~c+ V \sqrt{\Lambda} \alpha+A^\top \beta
    \end{array}\right\}, \notag
\end{align}
and $\ind_n(\mathcal{J}) \in \{0,1\}^n$ is a binary vector whose coordinates contained in the set~$\mathcal{J}$ are ones. 
\end{lemma}
We note that the set of optimal indices defined in~$\mc J(\alpha, \beta)$ may not be unique. In such scenarios, we adopt a deterministic tiebreaker (e.g., a lexicographic rule) to introduce~$\mc J$ as a proper single-valued function. The observation in~\eqref{eq:J-set} is the key building block for two scalable optimization algorithms that we will propose to tackle problem~\eqref{eq:P_k}.

\subsection{Alternating best response}
\label{sec:br}

The first proposed algorithm can be cast as an attempt to find a saddle point (a Nash equilibrium) of program~\eqref{eq:M_k}, if it exists. {This is similar to the approach discussed in~\citet[Theorem~3.2.3]{ref:bertsimas2020solving}, which concerns the different problem of sparse PCA. Given the inherent nonconvexity of the problem due to the binary variables~$z$, such an equilibrium may not exist. In that case, the algorithm will not converge. Nonetheless, we will also discuss how this approach can still be viewed as a ``safe screening" scheme~\cite{ref:atamturk2020safe}. 

Given $z \in \{0,1\}^n$, we define the function 
\begin{align}\label{eq:BR}
    \BR(z) \Let \arg \max_{\substack{\alpha\in\R^k\\\beta\in\R^m_+}} L(z, \alpha, \beta). 
\end{align}
If the function~$L$ in~\eqref{eq:BR} does not have a unique optimizer, in a similar fashion as as $\mc J$ defined in Lemma~\ref{lem:closed form}, we deploy a deterministic tiebreaker to properly introduce a single-valued function $\BR$. The function $\BR$ is the maximizer of the loss function for a fixed value of $z$, to which we refer as the ``best response". While the objective function~$L$ is a jointly concave quadratic function in the variables $(\alpha,\beta)$, the description of the optimizer does not necessarily have an explicit description due to the constraint $\beta \ge 0$. However, in the absence of the constraint~$Ax \le b$ in \eqref{eq:P_k}~(e.g., $A = 0, b = 0$), the function~\eqref{eq:BR} can be explicitly described as
\begin{align}\label{eq:BR-explicit}
    \BR(z) =  - (I_k/\eta +  \sqrt{\Lambda} V^\top \diag(z) V \sqrt{\Lambda})^{-1} \times\\  \sqrt{\Lambda} V^\top \diag(z)c.\notag
\end{align}
We note that we slightly abuse the notation as the explicit description~\eqref{eq:BR-explicit} is indeed the fist element~($\alpha$-component) of the original definition~\eqref{eq:BR}.

\begin{proposition}[Alternating best response]
\label{prop:BR}
Consider the set of update rules 
\begin{align}\label{alg:BR}
\left\{
\begin{array}{cl}
    \begin{bmatrix} \alpha_{t+1}  \\ \beta_{t+1} \end{bmatrix} & = \BR(z_t)\vspace{1mm}\\
    z_{t+1} & = \ind_n\big(\mathcal{J}(\alpha_{t},\beta_{t})\big),
    \end{array}\right. 
\end{align}
where the set~$\mathcal{J}$ and the function~$\BR$ are defined as in Lemma~\ref{lem:closed form} and \eqref{eq:BR}, respectively. Starting from an initialization~$(z_0,\alpha_0,\beta_0)$, algorithm~\eqref{alg:BR} converges 
after finitely many iterations to a limit cycle. If the set of this period behavior is singleton (i.e., the iterations convergence to a fixed point), then the variable $z$ of the convergence point is the optimal solution of program~\eqref{eq:M_k}. 
\end{proposition}

\begin{proof}
Recall that for any~$z \in \{0, 1\}^n$ the objective function~$L$ in \eqref{eq:BR} is strongly concave, and that admits a unique maximizer. On the other hand, the number of possible binary variable~$z$ is finite and bounded by~$2^n$. These two observations together then imply that the iterations~\eqref{alg:BR} necessarily yield a period behavior with the cardinality at most~$2^n$. When the period is one, it then means that the best response algorithm has an equilibrium, implying that the min-max program~\eqref{eq:M_k} is indeed a minimax game with a Nash equilibrium. Namely, there exist~$(\alpha\opt,\beta\opt) \in \R^k\times\R_+^m$ and $z\opt \in \{0,1\}^n$ such that for all $(\alpha, \beta)\in \R^k\times\R_+^m$, and $z \in \{0,1\}^n$, $\sum z_j \le s$, we have
\begin{align*}
L(z\opt,\alpha,\beta) \le L(z\opt,\alpha\opt,\beta\opt) \le L(z,\alpha\opt,\beta\opt)
\end{align*}
and, by definition, $z\opt$ solves the outer minimization of program~\eqref{eq:M_k}.
\end{proof}

In the iterative scheme~\eqref{alg:BR}, evaluating the best response function~$\BR(z_t)$ is equivalent to solving a linearly constrained convex quadratic program, which can be done efficiently using commercial solvers such as MOSEK~\cite{mosek}. In case we have no linear constraints in the form of~$Ax \le b$, we can also use the explicit description~\eqref{eq:BR-explicit}. Therefore, the algorithm~\eqref{alg:BR} is indeed highly tractable. 

\begin{remark}[Safe screening]\label{rem:safe}
    We expect that the periodic behavior anticipated by Proposition~\ref{prop:BR} typically has a periodicity larger than one. In fact, if the min-max characterization in Proposition~\ref{prop:minmax} is not interchangeable without suffering from a duality gap, 
    then the periodic behavior does have more than one element. 
    In this setting, one can consider all the indices $\mathcal{J}(\alpha_t,\beta_t)$, where $(\alpha_t,\beta_t)$ belongs to the period behavior, as potential candidates for the ones of the optimal vector~$z^\star$ in~\eqref{eq:M_k}. This selection is indeed in accordance with the safe screening terminology of~\citet{ref:atamturk2020safe}. 
\end{remark}

\subsection{Dual program: a subgradient ascent approach}
\label{sec:subgradient}

The second proposed algorithm aims to solve the dual of program~\eqref{eq:M_k} described via
\begin{align}\label{eq:D_k}
\tag{$\mathcal{D}_k$}
    \d_k \Let \max\limits_{\substack{ \alpha \in \mathbb{R} ^{k} \\\beta \in \R_+^m}}~\min\limits_{\substack{z \in \{0, 1\}^n \\ \sum z_j \le s}}~L(z, \alpha, \beta),
\end{align}
where the function~$L$ was defined in~\eqref{eq:L}. Thanks to the weak duality, it is obvious that $\d_k \le \J_k$. The second approach is essentially the application of the subgradient ascent algorithm to the inner minimal function 
\begin{align}\label{eq:f}
    f(\alpha,\beta) \Let \min\limits_{{z \in \{0, 1\}^n, \sum z_j \le s}}~L(z, \alpha, \beta).
\end{align}
We note that the continuous relaxation of the binary variable $z$ from $\{0,1\}^n$ to $[0,1]^n$ in \eqref{eq:f} does not change anything and the program remains equivalent to the original program~\eqref{eq:D_k}.
Also notice that the function $f$ in~\eqref{eq:f} is concave and piecewise quadratic, jointly in~$(\alpha,\beta)$. This observation allows us to apply the classical subgradient algorithm from the convex optimization literature~\citep[Section~3.2.3]{ref:nesterov2003introductory}.

\begin{proposition}[Dual program]\label{prop:DA} Consider the set of update rules defined as
\begin{align}\label{alg:dual}
    \left\{\begin{array}{l}
        \alpha_{t+1} =  (1 - \frac{1}{2}\kappa_t)\alpha_t -  \frac{1}{2}\eta\kappa_t \sqrt{\Lambda} V^\top \diag(z_t) \times \vspace{1mm}\\
         \qquad \qquad\qquad\qquad(c + V \sqrt{\Lambda} \alpha + A^\top \beta) \vspace{1mm}\\
         \beta_{t+1} =  \max\Big\{0, \beta_t -\kappa_t b - \frac{1}{2}\eta\kappa_t A \diag(z_t) \times \vspace{1mm}\\
         \qquad\qquad\qquad\qquad(c +  V \sqrt{\Lambda} \alpha + A^\top \beta)\Big\} \vspace{1mm}\\
         z_{t+1} = \ind_n\big(\mathcal{J}(\alpha_{t+1},\beta_{t+1})\big), 
    \end{array}\right. 
\end{align}
where $\{\kappa_t\}_t$ is the sequence of step sizes that satisfy the non-summable diminishing rule\footnote{ \label{footnote:stepsize} For instance, $\kappa_{t}=a/\sqrt{t}$ for a constant $a \in \R_+$.}
\[
\lim _{t \rightarrow \infty} \kappa_{t}=0, \quad \sum_{t=1}^{\infty} \kappa_{t}=\infty. 
\]
Then, algorithm~\eqref{alg:dual} converges to the optimal value~$d_k^*$ of problem~\eqref{eq:D_k}, i.e.,
\[ \d_k = \lim_{t \to \infty} f(\alpha_t,\beta_t).\]
Moreover, if the variable~$z_t$ also converges, then the convergent binary variable is the solution to program~\eqref{eq:M_k} and the duality gap between \eqref{eq:M_k} and \eqref{eq:D_k} is zero, i.e.,  $\d_k = \J_k$. 
\end{proposition}

\begin{proof}
Using standard results in variational analysis~\cite{ref:rockafellar2009variational}, a subgradient of the function~$f(\alpha,\beta)$ can be computed as
\begin{align*}
    \frac{\partial f}{\partial \alpha} &= - \frac{1}{2} \alpha -  \frac{\eta}{2} \diag(\sqrt{\lambda}) V^\top \diag(z_{\alpha,\beta})\times\\& \qquad \qquad \qquad (c + V \diag(\sqrt{\lambda}) \alpha + A^\top \beta), \\
    \frac{\partial f}{\partial \beta} &= - b -\frac{\eta}{2} A \diag(z_{\alpha,\beta}) (c +  V \diag(\sqrt{\lambda}) \alpha + A^\top \beta),
\end{align*}
where $z_{\alpha,\beta} = \ind_n\big(\mathcal{J}(\alpha,\beta)\big)$ is the optimizer of the objective function~$L(.,\alpha,\beta)$; see also \eqref{eq:J-set}. The subgradient algorithm updates the dual variables $\alpha_t,\beta_t$ by the following rule
\begin{align*}
    \begin{bmatrix}\alpha_{t+1} \vspace{1mm}\\ \beta_{t+1} \end{bmatrix} = \begin{bmatrix}\alpha_{t} \vspace{1mm}\\ \beta_{t} \end{bmatrix} + \kappa_t \begin{bmatrix}\frac{\partial}{\partial \alpha}f(\alpha_t,\beta_t) \vspace{1mm}\\ \frac{\partial}{\partial \beta}f(\alpha_t,\beta_t)\end{bmatrix},
\end{align*}
where $\kappa_t$ is the learning rate or stepsize. The computational complexity of the subgradient algorithm is well-known for concave and Lipschitz-continuous objective functions~\cite{ref:boyd2003subgradient,ref:nesterov2003introductory}. In the remainder of the proof, we first verify that the classical results are applicable here too. To continue, let $\lambda = (\alpha,\beta)$ for short. Note that $z \in \{0, 1\}^n$ which bounds the number of quadratic regions of $f$ to $2^n$ which, in turn, implies that any level set of $f$ is bounded. Since $f$ is also concave and thus continuous, we conclude that $f$ is in fact Lipschitz-continuous in any fixed level set of $f$. At iteration $t$, say $f(\lambda_t) = f_{z_t}(\lambda_t) := L(z_t,\alpha_t,\beta_t)$. This also means $g_t = \nabla f_{z_t}(\lambda_t)$ is a subgradient of $f$ at $\lambda_t$. If $t$ is large enough, then step size $\kappa_t$ is small enough, and therefore the algorithm update $\lambda_{t+1} = \lambda_t + \kappa_t g_t$ increases the value of $f_{z_t}$. Hence, $f(\lambda_t) = f_{z_t} (\lambda_t) < f_{z_t}(\lambda_{t+1}) \le  f(\lambda_{t+1})$. Therefore, the subgradient algorithm remains afterwards within a level set of $f$, specified by $\{ \lambda: f(\lambda) \ge f(\lambda_t)\}$ when $t$ is sufficiently large. Now recall that $f$ is Lipschitz-continuous in any fixed level set of $f$. So, we can apply the result of \citet{ref:boyd2003subgradient} for all sufficiently large $t \in \mathbb N$. 

Concerning the second part of the assertion, suppose that the variable~$z_t$ also converges to a binary variable~$z\opt$. Since the feasible set of the variable~$z$ is finite, the convergence assumption effectively implies that for all sufficiently large~$t \in \mathbb N$, we have constant~$z_t = z\opt$ in~\eqref{alg:dual}. As such, the dual algorithm~\eqref{alg:dual} essentially reduces to 
\begin{align*}
      \begin{bmatrix}\alpha_{t+1} \vspace{1mm}\\ \beta_{t+1} \end{bmatrix} = \begin{bmatrix}\alpha_{t} \vspace{1mm}\\ \beta_{t} \end{bmatrix} + \kappa_t \begin{bmatrix}\frac{\partial}{\partial \alpha}L(z\opt,\alpha_t,\beta_t) \vspace{1mm}\\ \frac{\partial}{\partial \beta}L(z\opt,\alpha_t,\beta_t)\end{bmatrix}.
\end{align*}
The above argument allows us to also interpret the algorithm~\eqref{alg:dual} as the subgradient ascent algorithm for the quadratic concave mapping~$(\alpha,\beta) \mapsto L(z\opt,\alpha,\beta)$. This observation yields 
\begin{align}\label{eq:lim-max}
    \lim_{t \to \infty} L(z\opt,\alpha_t,\beta_t) = \max_{\alpha \in \R^k, \, \beta \in \R^m_+} L(z\opt, \alpha,\beta). \notag
\end{align}
Thanks to the above equality, one can inspect that
\begin{align*}
    \J_k \ge \d_k & = \lim_{t \to \infty} L(z\opt,\alpha_t,\beta_t) =  \max_{\substack{\alpha \in \R^k \\ \beta \in \R^m_+}} L(z\opt, \alpha,\beta) \ge \J_k,
\end{align*}
where the first inequality is due to the weak duality between programs~\eqref{eq:M_k} and \eqref{eq:D_k}, and the last inequality follows from the definition of the optimal value of~\eqref{eq:M_k}. Since both sides of the above inequalities are $\J_k$, all the middle terms coincide. Thus, it concludes that~$z\opt$ solves program~\eqref{eq:M_k} and the zero duality gap holds, i.e., $\J_k = \d_k$.
\end{proof}

Similar to the best response algorithm in Proposition~\ref{prop:BR}, we expect that in the long-run the duality algorithm~\eqref{alg:dual} exhibits a period behavior over a number of $z_t \in \{0,1\}^n$. In this light, one can also consider the coordinates of ones elements of $z_t$ as a safe screening suggestion (cf.~Remark~\ref{rem:safe}).

\begin{remark}[Computational complexity]
    Formulating $(\mathcal P_k)$ requires a PCA decomposition with a (crude) time complexity of $\mathcal{O}(n^3)$, and sorting for operation~\eqref{eq:J-set} takes $\mathcal O(n \log n)$. In addition, the best response algorithm in~\eqref{alg:BR} has a complexity of $\mathcal{O}(k^3 + nk)$. Thus, overall the best response method has a complexity of $\mathcal{O}(n^3 + t(k^3 + nk + n \log n))$, where $t$ are the number of iterations. The dual program algorithm~\eqref{alg:dual} requires algebraic operations with complexity $\mathcal{O}(nk)$, sand that the overall complexity of the dual program is $\mathcal{O}(n^3 + t(nk + n \log n))$.
\end{remark}

\subsection{Post-processing}
\label{sec:post}
Suppose, without any loss of generality, that either the best response algorithm~\eqref{alg:BR} or the dual program algorithm~\eqref{alg:dual} terminates after $T$ iterations. By collecting the incumbent solutions $z_{T-p}, \ldots, z_{T}$ in the variable $z$ over the last $p$ iterations,  we can form the unique indices 
\[
    Z = z_{T-p}~|~z_{T-p +1}~| ~\cdots~ | ~z_T \in \{0, 1\}^n,
\]
where $|$ represents the componentwise OR operator. Intuitively, the binary value of $Z_i$ indicates if at least one of the last $p$ incumbent solutions has the $i$-th element being non-zero. The vector $Z$ thus represents the indices of $x$ that are likely to be non-zero in the optimal solution of problem~\eqref{eq:P}. We now utilize the binary vector $Z$ as an input to resolve the reduced problem
\begin{equation} \label{eq:P_Z} \tag{$\mathcal P_{Z}$}
\begin{array}{cl}
    \min & {\langle c, x\rangle}+\langle x, Q x\rangle + \frac{1}{\eta}\|x\|_2^2\\
    \St & x\in\R^n,\\
    &A x \leq b,~ \|x\|_0 \leq s,~ |x| \le MZ,
\end{array}
\end{equation}
where $M$ is the big-$M$ constant. If $Z_i=0$, then the last constraint of~\eqref{eq:P_Z} implies that $x_i = 0$ and a preprocessing step can remove this redundant component in $x$. As a consequence, the effective dimension of the variable $x$ in~\eqref{eq:P_Z} is upper bounded by the number of non-zero elements in $Z$, which is essentially $\|Z\|_0$. It is likely that $Z$ has many elements that are 0, thus $\|Z\|_0 \ll n$ and problem~\eqref{eq:P_Z} is easier to solve compared to~\eqref{eq:P}. In general, $\| Z\|_0 > s$, so \eqref{eq:P_Z} remains a binary quadratic optimization problem. In the optimistic case when $\| Z \|_0 = s$, then the cardinality constraint $\|x\|_0 \le s$ becomes redundant and \eqref{eq:P_Z} reduces to a quadratic program.


In practice, the magnitude of $M$ may affect the run time, and a tight value of $M$ can significantly improve the numerical stability and reduce the solution time. We follow the suggestion from~\citet{ref:hazimeh2020sparse} to compute $M$ as follows. Given the terminal solution $z_T$ from either the best response algorithm~\eqref{alg:BR} or the dual program algorithm~\eqref{alg:dual}, we solve problem~\eqref{eq:P_Z} with the input $Z$ being replaced by $z_T$ to get $x_T$. As $z_T$ satisfies~$\|z_T\|_0 = s$, this problem reduces to a quadratic program. To ensure that the big-$M$ formulation adds no binding constraints we assign $M = 4 \|x_T\|_{\infty}$. 

Lastly, we emphasize that the solution $z_t$ in the subgradient ascent algorithm~\eqref{alg:dual} does not fluctuate significantly from one iteration to another. Thus, for the dual program approach, we need to set a periodic value $p$ which is sufficiently large in order to recuperate meaningful signals on the indices. The best response method using the update~\eqref{alg:BR}, on the contrary, requires a smaller number of period $p$. 

\section{Numerical Experiments}
\label{sec:numerical}
We benchmark different approaches to solve problem \eqref{eq:P} in the sparse linear regression setting. All experiments are run on a laptop with Intel(R) Core(TM) i7-8750 CPU and 16GB RAM using MATLAB 2020b. The optimization problems are modeled using YALMIP~\cite{ref:lofberg2004yalmip}, as the interface for the mixed-integer solver \cite{mosek}. Codes are available at: \url{https://github.com/RVreugdenhil/sparseQP}


We compare our algorithms against four state-of-the-art approaches. They include two screening methods: the safe \texttt{screening} method of \citet{ref:atamturk2020safe} and the \texttt{warm start} method in~\citet{ref:bertsimas2017sparse}, and two direct optimization approaches: the Algorithm 7 in \citet{ref:beck2015on} (denoted \texttt{BH Alg 7}) and the method of \citet{ref:ganzhao2020block} (denoted \texttt{KDD}).\footnote{Available at \url{https://yuangzh.github.io/}}
The best response alternation in Section~\ref{sec:br} is referred to as \texttt{BR}, and the dual program approach in Section~\ref{sec:subgradient} is referred to as \texttt{DP}.



\subsection{Empirical results for synthetic data}\label{sec:syn}
We generate the covariate $\omega \in \R^n$ and the univariate response $\xi \in \R$ using the 
linear model
\[
\xi = x_{\mathrm{true}}^\top \omega + \epsilon
\]
following the similar setup in~\citet{ref:bertsimas2017sparse}. The unobserved true vector
$x_{\mathrm{true}} \in \R^n$ has $s$-nonzero components at indices selected uniformly at random, without replacement.
The nonzero components in $x_{\mathrm{true}}$ are selected uniformly at random from $\{\pm 1\}$. Moreover, the covariate $\omega$ is 
independently generated from a Gaussian distribution $\mc N(0,\Sigma)$, where $\Sigma$ is parametrized by the correlation coefficient $\rho$ as $\Sigma_{i,j} \Let \rho^{|i-j|}$ for all $i,  j \in [n]$ and $0 \le \rho \le 1$. The noise $\epsilon$ is independently generated from a normal distribution $\mc N(0,\sigma^2)$ with
\[
    \sigma^2 = \frac{\text{var}(x_{\mathrm{true}}^\top \omega)}{\text{SNR}}
     = \frac{x_{\rm true}^\top \Sigma x_{\rm true}}{\text{SNR}},
\]
where SNR is a chosen signal-to-noise ratio \cite{ref:xie2020scalable}. 

To avoid a complicated terminating criterion, we run the \texttt{BR} method for $T_{\texttt{BR}}= 20$ iterations, and we run the \texttt{DP} method for $T_{\texttt{DP}} = 500$ iterations. We empirically observe that \texttt{BR} converges by $T_{\texttt{BR}}=20$ on the synthetic data. The stepsize constant for \texttt{DP} (see Footnote~\ref{footnote:stepsize}) is set to $a = 4 \times 10^{-3}$. 
Regarding the post-processing step in Section~\ref{sec:post}, we fix $p_{\texttt{BR}}= 6$ and $p_{\texttt{DP}} = 50$ as the number of terminating solutions that are used to estimate the support $Z$. For the experiment on the synthetic data, the big-$M$ constant is set to $4$, because $\|x_{\mathrm{true}}\|_{\infty} = 1$.

Our first experiment studies the impact of the regularization parameter $\eta$ on the performance of \texttt{BR} and \texttt{DP} in terms of safe screening. To this end, we fix $N = 1000, n = 1000, s = 10$, $\rho = 0.5$ and $\text{SNR}= 6$ to generate the data. 

The screening capacity of \texttt{BR} and \texttt{DP} is measured by the sparsity of the input parameter $Z$ in problem~\eqref{eq:P_Z}, this reduced size is measured by $\| Z\|_0$. A similar quantity can be computed for \texttt{screening}. We choose the dimension of subspace $k=400$ to ensure that our the principal component approximation generates a good quality solution to the original problem. Table~\ref{tab:syneta} reports the screening capacity and we can observe that \texttt{screening} effectively reduces the dimension for small values of $\eta$, which is in agreement with the empirical results reported in~\citet{ref:atamturk2020safe}. However, \texttt{screening} performs less convincingly for $\eta \ge 1$. Our methods \texttt{BR} and \texttt{DP} perform more consistently over the whole range of $\eta$: they can reduce \eqref{eq:P_Z} to a quadratic program for 59\% and 81\% of all instances respectively.

\begin{table}[h!]
\caption{Effective problem size measured by $\|Z\|_0$ for varying $\eta$, averaged over 25 replications. Lower is better. Values are rounded to nearest integer, asterisks denote that ~$\|Z\|_0\!\!=\!\!s$ on all instances.}
\label{tab:syneta}
\begin{tabular}{|l|l|l|l|}
\hline
               & \texttt{DP} $k=400$& \texttt{BR} $k=400$& \texttt{screening} \!\!\!     \\ \hline
$\eta =   10^2$ & 69    & 20      & 1000  \\ \hline
$\eta = 10$    & 10$^\star$      & 20    & 1000 \\ \hline
$\eta = 1$     & 10      & 10      & 809   \\ \hline
$\eta = 10^{-1}$   & 10$^\star$      & 10      & 463 \\ \hline
$\eta =10^{-2}$   & 10$^\star$      & 10$^\star$      & 45 \\ \hline
$\eta =10^{-3}$  & 10$^\star$      & 10$^\star$      & 10  \\ \hline
\end{tabular}
\end{table}

We also study the performance of our methods in a setting with $\rho \ge 0.7$. We measure the quality of the estimator $x$ by the mean squared error on the data
\[
    \text{MSE} = \displaystyle \frac{1}{N} \sum_{i=1}^N \|\xi_i - x^\top \omega_i\|_2^2.
\]
We fix the regularisation term $\eta = 10$. Setting $\eta$ to a lower value would cause unwanted shrinkage of the estimator $x$, which increases the MSE: instances where $\eta=0.1$ reported a MSE for all methods of at least 5 times larger than that of the MSE using $\eta = 10$. As seen in Table~\ref{tab:syneta} for this specific $\eta$ the \texttt{screening} method does not reduce the problem size, so we compare to the MSE generated by \texttt{warm start}.    

We observe in Table~\ref{tab:synrho} that \texttt{DP} outperforms the other methods in terms of MSE with highly correlative data ($\rho \ge 0.8$). A possible explanation for why \texttt{DP} outperforms the \texttt{warm start} is that the eigenvalue decomposition can convert highly correlative features to independent principal components \cite{ref:liu2003principal}.  
\begin{table}[h!]
\caption{MSE over different $\rho$ averaged over 25 independent replications. Lower is better. }
\label{tab:synrho}
\begin{tabular}{|l|L{9 mm}|L{9 mm}|L{9 mm}|L{9 mm}|L{9 mm}|}
\hline
                  & \texttt{DP} $k=400$                    & \texttt{BR} $k=400$                    & \texttt{warm start}& \texttt{BH Alg 7}& \texttt{KDD}\!\!\\ \hline
                  $\rho =   0.7$ & 1.829  & 1.829   & 1.829 & 1.897 & 1.829                          \\ \hline
$\rho=0.8$ & 1.863    & 1.969    & 1.988 & 2.114 & 1.866                        \\ \hline
$\rho=0.9$ & 1.895    & 3.666    & 3.714 & 2.570 & 1.984                         \\ \hline
\end{tabular}
\end{table}

\subsection{Empirical results for real data}\label{sec:real}
We benchmark different methods using real data sets~\cite{ref:dua2017uci}; the details are listed in Table~\ref{tab:realdata}. We preprocess the data by normalizing each covariate and target response independently to values in the range $[0, 1]$. For each independent replication, we randomly sample 70\% of the data as training data and 30\% as test data. The training set and the test set have $N_{\rm train}$ and $N_{\rm test}$ samples, respectively. 

\begin{table}[h!]
\caption{List of UCI datasets used for experiments alongside their feature size $n$, total sample size $\bar N$ and the dimension of subspace $\hat k$, specified earlier in Figure~\ref{fig:Frobdecay}.
}
\label{tab:realdata}
\centering
\begin{tabular}{|l|rrr|}
\hline Name & $n$ & $\bar N$ & $\hat{k}$ \\
\hline 
Facebook (FB) & 52 & 199,030 & 14\\
OnlineNews (CR) & 58 & 39,644 & 20 \\
SuperConductivity (SC)& 81 & 21,263 & 9 \\
Crime (CR) & 100 & 1,994 & 53\\
UJIndoor (UJ)& 465 & 19,937 & 49\\ 
\hline
\end{tabular}
\end{table}

\begin{table*}[t!]
\scriptsize
\centering
\captionsetup{justification=centering}
\caption{In-sample MSE on real datasets, averaged over 50 independent train-test splits. Lowest error for each case is highlighted in grey.}
\label{tab:realinsample}
\begin{tabular}{|l|l|l|l|l|l|l|l|l|}
\hline &&&&&&\\[-1em]
     & \texttt{DP} $k=40$ & \texttt{DP} $k=\hat{k}$ & \texttt{BR} $k=40$ & \texttt{BR} $k=\hat{k}$ & \texttt{warm start} & \texttt{screening} & \texttt{BH Alg 7}& \texttt{KDD} \\ \hline
(FB) & 3.039$\times 10^{-4}$                            & 3.040$\times 10^{-4}$                                 & 3.036$\times 10^{-4}$                            & \cellcolor[HTML]{EFEFEF}\textbf{3.034}$\boldsymbol{\times 10^{-4}}$         & out of memory                        & 3.036$\times 10^{-4}$  & 3.220$\times 10^{-4}$                            & 3.436$\times 10^{-4}$                           \\ \hline
(ON) & \cellcolor[HTML]{EFEFEF}\textbf{1.912}$\boldsymbol{\times 10^{-4}}$    & 1.913$\times 10^{-4}$                                 & 1.914$\times 10^{-4}$                            & 1.914$\times 10^{-4}$                                 & 1.914$\times 10^{-4}$                             & 1.914$\times 10^{-4}$ & 1.921$\times 10^{-4}$                            & 1.922$\times 10^{-4}$                            \\ \hline
    (SC) & 1.262$\times 10^{-2}$                            & 1.263$\times 10^{-2}$                                 & 1.396$\times 10^{-2}$                            & 1.368$\times 10^{-2}$                               & 1.326$\times 10^{-2}$                             & \cellcolor[HTML]{EFEFEF}\textbf{1.256}$\boldsymbol{\times 10^{-2}}$ & 1.455$\times 10^{-2}$                            & 1.470$\times 10^{-2}$    \\ \hline
(CR) & 2.775$\times 10^{-2}$                            & 2.775$\times 10^{-2}$                                 & 2.807$\times 10^{-2}$                            & 2.801$\times 10^{-2}$                                 & 2.800$\times 10^{-2}$                             & \cellcolor[HTML]{EFEFEF}\textbf{2.760}$\boldsymbol{\times 10^{-2}}$ & 3.072$\times 10^{-2}$                            & 3.063$\times 10^{-2}$    \\ \hline
(UJ) & \cellcolor[HTML]{EFEFEF}\textbf{2.118}$\boldsymbol{\times 10^{-2}}$    & 2.294$\times 10^{-2}$                                 & 2.673$\times 10^{-2}$                            & 2.678$\times 10^{-2}$                                 & 2.440$\times 10^{-2}$                             & 2.267$\times 10^{-2}$  & 3.804$\times 10^{-2}$                             & 3.066$\times 10^{-2}$                          \\ \hline
\end{tabular}
\end{table*}
The number of iterations for the \texttt{BR} and \texttt{DP} are chosen to be $T_{\texttt{BR}}=40$ and $T_{\texttt{DP}}=5000$, respectively. The number of terminating solutions that are used to estimate the support $Z$ in the postprocessing step is fixed to $p_{\texttt{BR}}= 10$ and $p_{\texttt{DP}} = 100$ for the two methods. The stepsize constant (see Footnote~\ref{footnote:stepsize}) is set $a = 2 \times 10^{-3}$. Moreover, we fix the big-$M$ constant using the procedure described in Section~\ref{sec:post}.

For the real datasets, the number of samples $\bar N$ is sufficiently bigger than the dimension $n$, thus we set the ridge regularization parameter to $\eta ={\sqrt{N_{\rm train}}}$ so that the effect of regularization diminishes as $\bar N$ increases. We choose the sparsity level $s=10$, similar to Section~\ref{sec:syn}, and set the time limit for MOSEK to 300 seconds.

\begin{figure}[h!]
\centering
\begin{subfigure}[b]{0.5\textwidth}
   \includegraphics[width=1\linewidth+0.1cm, trim={0.8cm 10cm 0  10cm},clip]{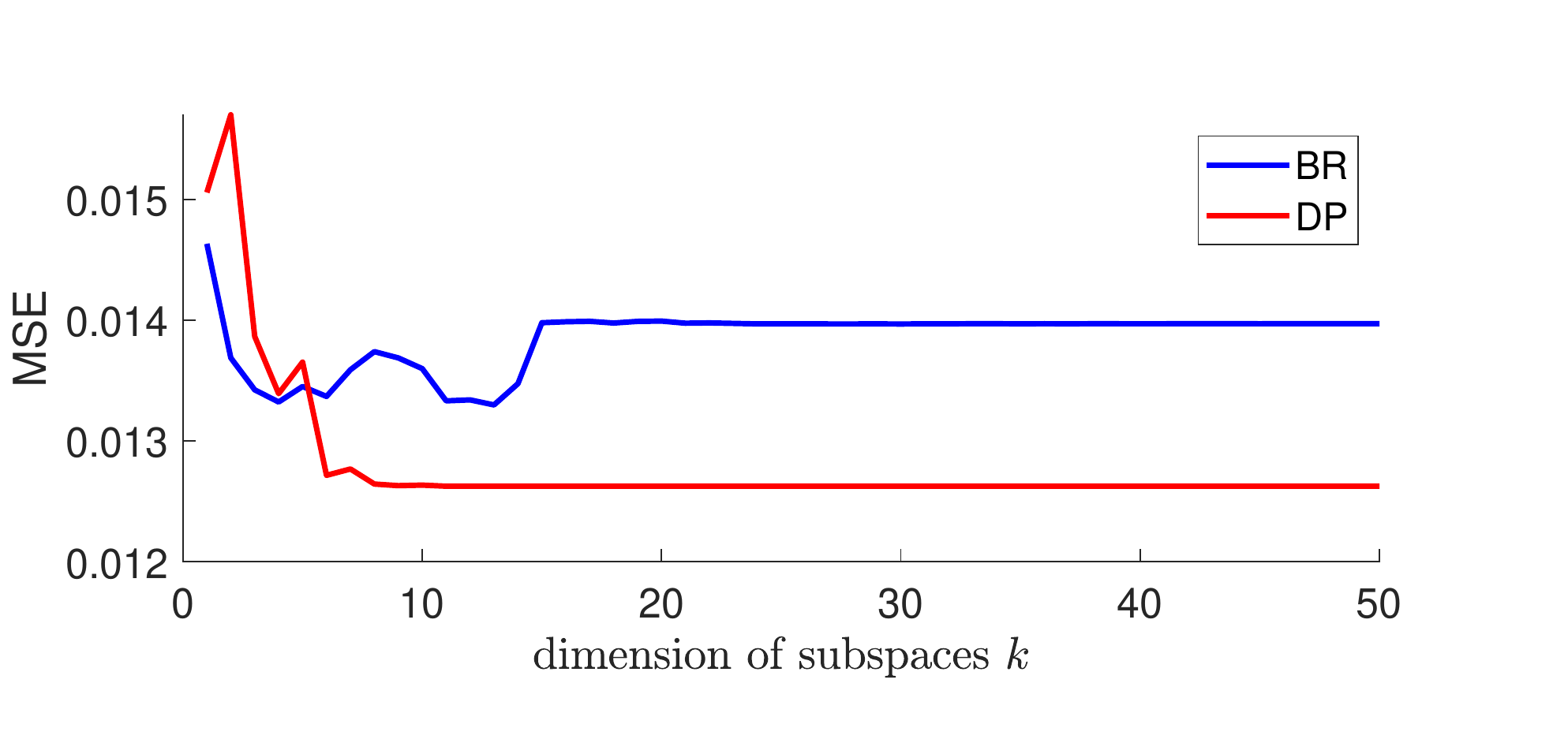}
   \vspace{-8mm}
   \caption{in-sample Mean Squared Error}
   \label{fig:kmse} 
\end{subfigure}
\vspace{-3mm}
\begin{subfigure}[b]{0.5\textwidth}
   \includegraphics[width=1\linewidth+0.1cm, trim={0.8cm 10cm 0  10cm},clip]{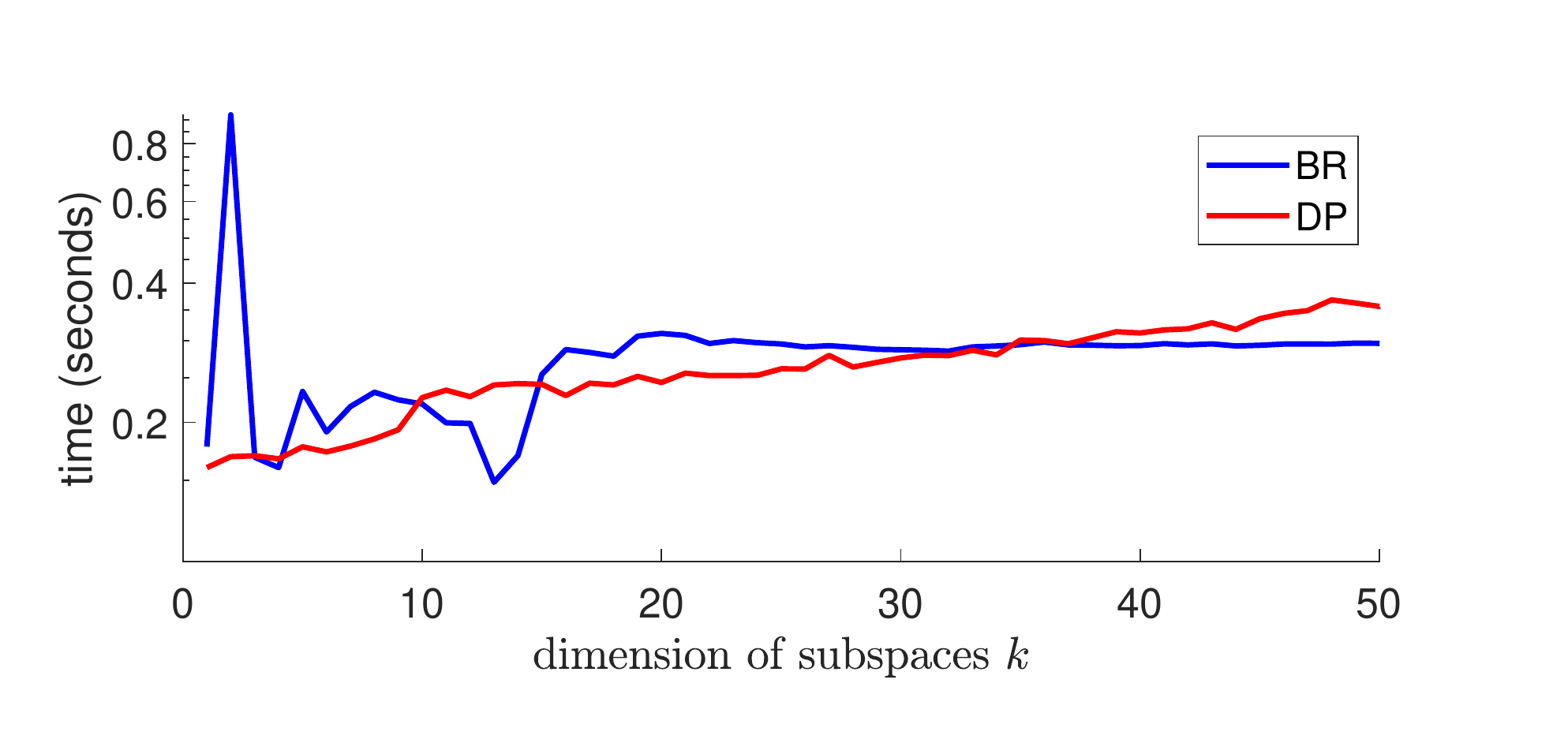}
   \vspace{-5mm}
   \caption{computational time}
   \label{fig:ktime}
\end{subfigure}
\caption{Effects of the dimension of subspaces $k$ on performance of \texttt{BR} and \texttt{DP} on the Superconductivity dataset. Results are averaged over 50 independent train-test splits.}
\label{fig:diffk}
\end{figure}

Figure~\ref{fig:diffk} illustrates that the MSE of \texttt{DP} monotonically decreases with the dimension of subspace for $k < 8$ and plateaus for $k\ge 8$.
The MSE of \texttt{BR} is non monotonic and converges when $k \ge 15$ 
even though the minimum is achieved at $k=13$. We define $k\opt$ as the (minimal) dimension of subspace $k$ corresponding to the lowest MSE. For the (SC) dataset, $k_{\texttt{BR}}\opt =13$ and $k_{\texttt{DP}}\opt =15$. Coincidentally, we observe that $k\opt$ is close to the value $\hat k$ reported in~Table~\ref{tab:realdata}. This observation also persists empirically for the other datasets.



Table~\ref{tab:realinsample} shows that \texttt{DP} delivers a lower in-sample MSE than \texttt{BR} in $4$ out of $5$ datasets, and \texttt{DP} also has a lower in-sample MSE than the \texttt{warm start}, \texttt{BH Alg 7} and \texttt{KDD} for all datasets. The \texttt{warm start} method runs out of memory for the (FB) dataset because it requires storing and computing based on a kernel matrix $K = [\omega_i^\top \omega_j]_{i,j}$ of dimension $N\times N$. This is in stark contrast to our proposed approach that computes only a matrix $Q$ of dimension $n\times n$, and then 
further 
reduce the computational burden by truncating the SVD of $Q$.
The memory usage of our method is hence not sensitive to the number of samples $N$. 

The \texttt{screening} method outperforms on the (SC) and (CR) dataset, however a careful examination of Table~\ref{tab:realred} shows that \texttt{screening} does minimal reduction effects for these two datasets. The result of \texttt{screening} in Table~\ref{tab:realinsample} on the (SC) and (CR) datasets is essentially the results obtained by applying the MOSEK solver to the original problem (reaching a time limit of 300 seconds). 

Table~\ref{tab:realred} also shows that our \texttt{DP} and \texttt{BR} methods can effectively reduce the number of effective variables. Our methods deliver a solution $x\opt$ in around 1 second for all datasets, including the time spent on computing the eigendecomposition of $Q$ and the solution time for solving~\eqref{eq:P_Z} using MOSEK.

\begin{table}[h]
\caption{Reduced problem size over different data rounded average over 50 independent train-test splits. Lower is better. }
\label{tab:realred}
\begin{tabular}{|l|l|l|l|}
\hline
                  & \texttt{DP} $k=\hat{k}$ & \texttt{BR} $k=\hat{k}$ & \texttt{screening}         \\ \hline

(FB)            & 11      & 20      & 49 \\ \hline
(ON)            & 15      & 20      & 57  \\ \hline
(SC)            & 11      & 20      & 77 \\ \hline
(CR)            & 12      & 20      & 100                        \\ \hline
(UJ)            & 16      & 20      & 465                       \\ \hline
\end{tabular}
\end{table}

\begin{remark}[Choice between \texttt{DP} and \texttt{BR}]
We have no theoretical or consistent numerical justification in favor of one of the proposed algorithms~\texttt{DP} or \texttt{BR}. However, Algorithm~\texttt{BR} in Proposition~\ref{prop:BR} typically converges faster (Fig. A in supplementary) while Algorithm~\texttt{DP} offers a better solution (Fig.~\ref{fig:kmse} and Tables~\ref{tab:synrho},\,\ref{tab:realred}). We thus suggest \texttt{DP} and \texttt{BR} as complementary approaches to solve the problem.
\end{remark}

\section{Acknowledgement}
We would like to thank the meta-reviewer and four anonymous reviewers for their constructive comments that helped improve the presentation of this paper. This research is partially supported by the ERC grant TRUST-949796.

\bibliography{main_and_supp.bbl}
\bibliographystyle{icml2021}

\appendix

\clearpage
\onecolumn

\section{Additional Numerical Experiments}
\subsection{Comparison of Computational Time to \texttt{warm start}}
We study the impact of the sample size $N$ on the recovery quality of the solution. We fix $n = 1000$, $s = 10$, $\rho = 0.5$, $\text{SNR}= 6$ and $\eta = 10$. 
We showcase the computational time of our methods and of the \texttt{warm start} in Figure~\ref{fig:syntime}, the computational time is defined as the time needed to generate $x\opt$. 
Note that the \texttt{BR} method uses MOSEK to obtain the solution to $x\opt$ because it does not converge to a single set $z$ for $\eta = 10$, so the solver time is also included in the computational time. We run the \texttt{BR} method for $T_{\texttt{BR}}= 20$ iterations, and we run the \texttt{DP} method for $T_{\texttt{DP}}= 500$ iterations.

\begin{figure}[h!]
\centering
   \includegraphics[width=12cm, trim={1cm 10cm 0  10cm},clip]{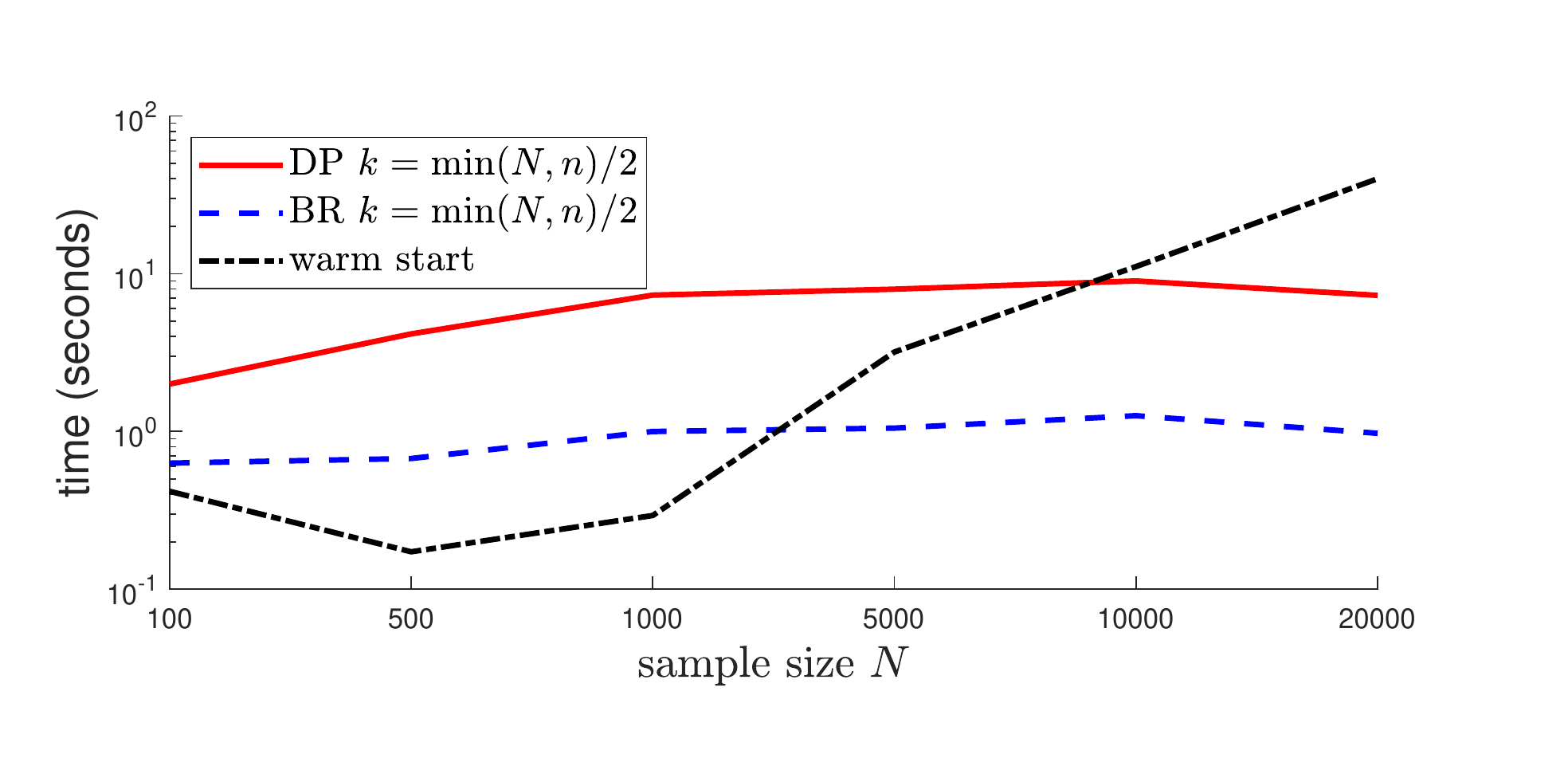}
   \caption{Computational time over different sample sizes averaged over 25 replications}
\label{fig:syntime} 
\end{figure}

We observe that the computational time of the \texttt{DP} method increases monotonically with the sample size $N$. Note that $T_{\texttt{BR}} \ll T_{\texttt{DP}}$ so calculating $Z_{\texttt{BR}}$ requires less time than $Z_{\texttt{DP}}$. We observe that when $N = 100$ the \texttt{BR} and the \texttt{warm start} have a higher computational time than for $N=500$. For the \texttt{BR}, this is 
because the number of non-zero elements in $Z$ (i.e., $\|Z\|_0$) is larger for $N=100$ than for $N=500$, hence MOSEK takes more time for $N=100$. The MSE of all methods is similar when $N \ge 500$, when $N = 100$ the MSE of all methods differs significantly at every instance. This is also observed by \cite{ref:bertsimas2017sparse}, which states that the computational time and MSE deteriorate as $N$ gets smaller relative to $n$.

We observe that \texttt{BR} and \texttt{DP} perform particularly well in terms of computational time in ranges where $N > n$ compared to the \texttt{warm start}. The running time of our method is less susceptive to the number of samples $N$. This is in stark contrast to the \texttt{warm start}, in which the kernel matrix of dimension $N$-by-$N$ is stored. 


\subsection{Comparison for Different SNR and $s$}
We extend the comparison made in the paper for different values of SNR and $s$. 
\begin{table}[h!]
\centering
\caption{MSE over different SNR averaged over 25 independent replications. Lower is better.}
\label{tab:synSNR}
\begin{tabular}{|l|l|l|l|l|l|}
\hline
           & \texttt{DP}   $k=400$ & \texttt{BR}   $k=400$ & \texttt{warm   start} & \texttt{Beck Alg 7}& \texttt{KDD}\\ \hline
$\text{SNR} =20$    & 0.588                                & 0.588                                & 0.588& 0.588& 0.588                                  \\ \hline
$\text{SNR} = 6$    & 1.767                               & 1.767                                & 1.767& 1.767& 1.767                                  \\ \hline
$\text{SNR} = 3$    & 3.452                                & 3.452                                & 3.452& 3.452& 3.452                                 \\ \hline
$\text{SNR} = 1$    & 10.190                               & 10.190                               & 10.190 & 10.198 & 10.205                                   \\ \hline
$\text{SNR} = 0.05$ & 194.592                              & 194.561                              & 194.560& 194.756& 199.928                               \\ \hline
\end{tabular}
\end{table}
\begin{table}[h!]
\centering
\caption{MSE over different $s$ averaged over 25 independent replications. Lower is better.}
\label{tab:syns}
\begin{tabular}{|l|l|l|l|l|l|}
\hline
       & \texttt{DP}   $k=400$ & \texttt{BR}   $k=400$ & \texttt{warm   start}& \texttt{Beck Alg 7}& \texttt{KDD} \\ \hline
$s= 5$   & 0.887                                & 0.887                               & 0.887 & 0.887 & 0.887                                 \\ \hline
$s = 10$ & 1.767                               & 1.767                                & 1.767& 1.767& 1.767                                  \\ \hline
$s = 20$ & 3.435                                & 3.435                                & 3.435 & 3.557 & 3.450                                  \\ \hline
$s = 30$ & 5.050                                & 5.050                                & 5.058  & 5.888 & 5.440                                  \\ \hline
$s = 40$ & 6.919                                & 6.928                                & 6.918 & 8.290 & 8.560                                \\ \hline 
\end{tabular}
\end{table}
In Table~\ref{tab:synSNR} and~\ref{tab:syns} we observe that the MSE over different SNR and $s$ is very similar for all methods.
This is due to the fact that all methods find a similar support $z\opt$. Using this support all problems solve the same convex quadratic programming problem. We also observe that the reduced size $\|Z_{\texttt{BR}}\|_0 \approx 2s$ and $\|Z_{\texttt{DP}}\|_0 \approx s$. So as the problem in~\eqref{eq:P_Z} increases with $s$, MOSEK takes more time to solve~\eqref{eq:P_Z} and because $\|Z_{\texttt{BR}}\|_0 > \|Z_{\texttt{DP}}\|_0$ the \texttt{DP} is faster for large $s$.

\subsection{Real Datasets}
For the real datasets listed in the main paper, we present the out-sample MSE for the different methods in Table~\ref{tab:realoutsample}.
\begin{table}[h!]
\scriptsize
\centering
\captionsetup{justification=centering}
\caption{Out-sample MSE on real datasets, averaged over 50 independent train-test splits. Lowest error for each dataset is highlighted in grey.}
\label{tab:realoutsample}
\begin{tabular}{|l|l|l|l|l|l|l|l|l|}
\hline &&&&&&&&\\[-1em]
                  & \texttt{DP}   $k=40$ & \texttt{DP} $k=\hat{k}$ & \texttt{BR}   $k=40$ & \texttt{BR} $k=\hat{k}$ & \texttt{warm   start} & \texttt{screening}& \texttt{BH Alg 7} & \texttt{KDD} \\ \hline
(FB)          &$3.026\times 10^{-4}$                             & $3.025\times 10^{-4}$                             & $3.022\times 10^{-4}$                             & $\boldsymbol{3.020\times 10^{-4}}$    \cellcolor[HTML]{EFEFEF}                         &   out of memory                                       & $3.022\times 10^{-4}$ & $3.203\times 10^{-4}$ & $3.409\times 10^{-4}$                            \\ \hline
(ON)       & $\boldsymbol{1.796\times 10^{-4}}$ \cellcolor[HTML]{EFEFEF}                             &$ 1.797\times 10^{-4}$                             & $1.797\times 10^{-4}$                             & $1.797\times 10^{-4}$                             & $1.797\times 10^{-4}$                               & $1.797\times 10^{-4}$ & $1.803\times 10^{-4}$  & $1.803\times 10^{-4}$                            \\ \hline
(SC) & $1.263\times 10^{-2}$                             & $1.263\times 10^{-2}$                             & $1.398\times 10^{-2}$                             & $1.370\times 10^{-2}$                             & $1.326\times 10^{-2}$                               &  $\boldsymbol{1.257\times 10^{-2}}$   \cellcolor[HTML]{EFEFEF}  & $1.454\times 10^{-2}$   & $1.473\times 10^{-2}$                          \\ \hline
(CR)             & $2.892\times 10^{-2}$                             & $2.891\times 10^{-2}$                             & $2.893\times 10^{-2}$                             & $2.894\times 10^{-2}$                             & $2.900\times 10^{-2}$                               &  $\boldsymbol{2.868\times 10^{-2}}$  \cellcolor[HTML]{EFEFEF}   & $3.103\times 10^{-2}$      & $3.148\times 10^{-2}$                        \\ \hline
(UJ)               & $\boldsymbol{2.149\times 10^{-2}}$ \cellcolor[HTML]{EFEFEF}                         & $2.324\times 10^{-2}$                             & $2.684\times 10^{-2}$                             & $2.691\times 10^{-2}$                             & $2.468\times 10^{-2}$                               & $2.291\times 10^{-2} $        & $3.848\times 10^{-2} $        & $3.080\times 10^{-2} $                           \\ \hline
\end{tabular}
\end{table}

Similar to the in-sample MSE, Table~\ref{tab:realoutsample} shows that \texttt{DP} delivers a lower out-sample MSE than \texttt{BR} in $4$ out of $5$ datasets, and \texttt{DP} also has a lower out-sample MSE than the \texttt{warm start}, \texttt{BH Alg 7} and \texttt{KDD} for all datasets. The \texttt{screening} method outperforms the \texttt{DP} on the (SC) and (CR) dataset, however as explained in the main paper for $\eta = \sqrt{N_{train}}$ the result of \texttt{screening} in Table~\ref{tab:realoutsample} on the (SC) and (CR) datasets is essentially the results obtained by applying the MOSEK solver to the original problem (reaching a time limit of 300 seconds). 

\section{Proof of Proposition~3.1}
We provide the proof of Proposition~3.1, which is not included in the main paper.
\begin{proof}
Using the big-$M$ equivalent formulation, we have
\[
     \begin{array}{rcll}
\mathcal{J}_k^{\opt} = 
\min\limits_{\substack{z \in \{0, 1\}^n \\ \sum z_j \le s}}& \min & \ds \sum_{i=1} ^{k} \lambda_{i}  y_{i}^{2} + \langle c, x\rangle + \eta^{-1} \|x\|_2^2 \\
    &\St & x \in\mathbb{R}^n,~y \in \R^k \\ 
    && \sqrt{\lambda_{i} } y_{i}=\sqrt{\lambda_{i} }\left\langle v_{i}, x\right\rangle & i \in [k] \\
     && | x_j | \le M z_j & j \in [n] \\
     && A x \leq b.
\end{array}
\]
Fix a feasible solution for $z$ and consider the inner minimization problem. By associating the first two constraints with the dual variables $\alpha$ and $\beta$, the Lagrangian function is defined as
\begin{align*}
    \mc L(x, y, \alpha, \beta) &= \sum_{i=1} ^{k} \lambda_{i}  y_{i}^{2} + \langle c, x\rangle + \eta^{-1} \|x\|_2^2 +
     \sum_{i=1} ^{k} \alpha_{i} \sqrt{\lambda_{i} }\left(\left\langle v_{i}, x\right\rangle-y_{i}\right) +  \beta^\top (Ax - b) \\
    &= - \beta^\top b +  y^\top \Lambda y - \alpha^\top \sqrt{\Lambda} y 
     + \Big\langle c+ V \sqrt{\Lambda} \alpha + A^\top \beta, x\Big\rangle + \eta^{-1} \|x\|_2^2,
\end{align*}
in which $\Lambda = \diag\{\lambda_1, \cdots,\lambda_k\}$.
For any feasible solution $z$, the inner minimization problem is a convex quadratic optimization problem and we have
\[
    \mathcal{J}_k^{\opt} = \min\limits_{\substack{z \in \{0, 1\}^n \\ \sum z_j \le s}}~
    \max\limits_{\substack{ \alpha \in \mathbb{R} ^{k} \\\beta \in \R_+^m}}~L(z, \alpha, \beta),
\]
where the objective function $L$ is defined as
\begin{align*}
L(z, \alpha, \beta) := -\beta^\top b + \min_{y \in \R^k}~y^\top \Lambda y - \alpha^\top \sqrt{\Lambda} y
+\min_{\substack{x \in\mathbb{R}^n\\ |x_j| \le M z_j ~\forall j}}\Big\langle c+ V \sqrt{\Lambda} \alpha + A^\top \beta, x\Big\rangle + \eta^{-1} \|x\|_2^2.
\end{align*} 
We will reformulate the two optimization subproblems in the definition of $L$.
For any feasible 
pair~$\beta \in \R_+^m$ and $\alpha \in \R^k$, the subproblem over $y$ is an unconstrained convex quadratic optimization problem. The corresponding optimal solution for $y$ is
\[
    y\opt(\alpha, \beta) = \frac{1}{2}  (\sqrt{\Lambda})^{-1} \alpha.
\]
Consequently, the optimal value of the $y$-subproblem is given by
\[
\min_{ y \in \R^k} ~y^\top \Lambda y - \alpha^\top \sqrt{\Lambda} y = -\frac{1}{4} \|\alpha\|_2^2.
\]
Next, consider the $x$-subproblem. Let $\gamma := c+ V \sqrt{\Lambda} \alpha+A^\top \beta$ and let $\gamma_j$ denote  the $j$-th element of $\gamma$. The big-$M$ equivalent formulation for the $x$-subproblem admits the form
\begin{align}
    \min_{\substack{x \in \R^n \\ |x_j| \le M z_j~\forall j}}~\sum_{j=1}^n \gamma_j x_j + \frac{x_j^2}{\eta}  = \sum_{j=1}^n \min_{\substack{x \in \R^n \\ |x_j| \le M z_j~\forall j}}~ \gamma_j x_j + \frac{x_j^2}{\eta}  = &\sum_{j=1}^n - \frac{\eta}{4} \gamma_j^2 z_j, \notag
\end{align}
where the last equality exploits the fact that the optimal solution of $x_j$ is
\[
    x_j\opt(z_j) = \begin{cases}
        -\frac{\eta}{2} \gamma_j &\text{if } z_j = 1, \\
        0 &\text{if } z_j = 0.
    \end{cases}
\]
We thus have
\[
    L(z, \alpha, \beta) =-\beta^\top b -\frac{1}{4} \sum_{i=1}^k \alpha_i^2 - \sum_{j=1}^n  \frac{\eta}{4} \gamma_j^2 z_j,
\]
where $\gamma = c+ V \sqrt{\Lambda} \alpha+A^\top \beta$ and $\gamma_j$ is the $j$-th element of $\gamma$. Rewriting the summations using norm and matrix multiplications completes the proof.
\end{proof}

\section{Principal Component Hierarchy for Sparsity-Penalized Quadratic Programs} \label{sec:OtherFormulation}

The approach proposed in the main paper can be extended to solve the $\|\cdot\|_0$-penalized problem of the form
\begin{equation} \notag
\begin{array}{cl}
    \min & {\langle c, x\rangle}+\langle x, Q x\rangle + \eta^{-1} \|x\|_2^2 + \theta \| x\|_0\\
    \St & x\in\R^n,~A x \leq b
\end{array}
\end{equation}
for some sparsity-inducing parameter $\theta > 0$. The corresponding approximation using $k$ principal components of the matrix $Q$ is
\begin{equation} \label{eq:sparsity-penalty} \tag{$\mathcal{W}_k$}
\begin{array}{cll}
\mathcal{U}_k^{\opt} 
\Let 
&\min & \ds \langle c, x\rangle + \sum_{i=1} ^{k} \lambda_{i}  y_{i}^{2} +\eta^{-1} \|x\|_2^2 + \theta \|x\|_0\\
    &\St & x \in\mathbb{R}^n,~y \in \R^k \\
    && A x \leq b \\
    && \sqrt{\lambda_{i} } y_{i}=\sqrt{\lambda_{i} }\left\langle v_{i}, x\right\rangle \quad i \in [k].
\end{array}
\end{equation}
\begin{proposition}[Min-max characterization] \label{prop:minmax-pen}
For each $k \le n$, the optimal value of problem \eqref{eq:sparsity-penalty} is equal to
\[
    \mathcal{U}_k^{\opt} = \min\limits_{z \in \{0, 1\}^n }~
    \max\limits_{\substack{ \alpha \in \mathbb{R} ^{k} \\\beta \in \R_+^m}}~H(z, \alpha, \beta),
\]
where the objective function $H$ is defined as
\begin{align} \label{eq:H}
&H(z, \alpha, \beta) \Let \theta \sum_{j=1}^n z_j -\beta^\top b -\frac{1}{4} \| \alpha\|_2^2 -\frac{\eta}{4}  (c+ V \sqrt{\Lambda} \alpha + A^\top \beta)^\top \diag(z) (c+ V \sqrt{\Lambda} \alpha + A^\top \beta). 
\end{align}
\end{proposition}
\begin{proof}[Proof of Proposition~\ref{prop:minmax-pen}]
The sparsity-penalized principal component approximation problem can be rewritten using the big-$M$ formulation as
\begin{equation*}  
\begin{array}{cll}
\min\limits_{z \in \{0, 1\}^n}  & \min & \ds  \langle c, x\rangle + \sum_{i=1} ^{k} \lambda_{i}  y_{i}^{2} + \eta^{-1} \|x\|_2^2 + \theta \sum_{j=1}^n z_j \\
    & \St & x \in\mathbb{R}^n,~y \in \R^k \\ 
    && \sqrt{\lambda_{i} } y_{i}=\sqrt{\lambda_{i} }\left\langle v_{i}, x\right\rangle \quad i \in [k] \\
    && |x_j|\leq M z_j \hspace{16.4mm} j \in [n] \\
    && A x \leq b.
\end{array}
\end{equation*}
For any feasible solution $z$, the inner minimization problem is a convex quadratic optimization problem. By strong duality, we have the equivalent problem
\[
    \mathcal{U}_k^{\opt} = \min\limits_{z \in \{0, 1\}^n }~
    \max\limits_{\substack{ \alpha \in \mathbb{R} ^{k} \\\beta \in \R_+^m}}~H(z, \alpha, \beta),
\]
where the objective function $H$ is
\begin{align*}
H(z, \alpha, \beta) &= -\beta^\top b + \min_{y \in \R^k}~y^\top \sqrt{\Lambda} y - \alpha^\top \diag(\sqrt{\Lambda}) y +
 \min_{\substack{x \in\mathbb{R}^n\\ |x_j|\leq M z_j  ~\forall j}}\Big\langle c + V \diag(\sqrt{\Lambda}) \alpha + A^\top \beta, x\Big\rangle + \eta^{-1} \|x\|_2^2 + \theta \sum_{j=1}^n z_j.
\end{align*}

Following proposition 3.1 we can calculate the optimal values for $y\opt$ and $x\opt$.
Considering the $x$-subproblem, let $\gamma = c+ V \sqrt{\Lambda} \alpha+A^\top \beta$ and $\gamma_j$ be the $j$-th element of $\gamma$. The big-$M$ equivalent formulation for the $x$-subproblem admits the form
\begin{align}
    \min_{\substack{x \in \R^n \\ |x_j| \le M z_j~\forall j}}~\sum_{j=1}^n \gamma_j x_j + \frac{x_j^2}{\eta} + \theta z_j     = &\sum_{j=1}^n \min_{\substack{x \in \R^n \\ |x_j| \le M z_j~\forall j}}~ \gamma_j x_j + \frac{x_j^2}{\eta} + \theta z_j \notag \\
    = &\sum_{j=1}^n (- \frac{\eta}{4} \gamma_j^2 + \theta) z_j , \label{eq:knapsackd}
\end{align}
where the last equality exploits the fact that the optimal solution of $x_j$ is
\[
    x_j\opt(z_j) = \begin{cases}
        -\frac{\eta}{2} \gamma_j &\text{if } \frac{\eta}{4} \gamma_j^2  > \theta , \\
        0 &\text{if } \frac{\eta}{4} \gamma_j^2  \leq \theta.
    \end{cases}
\]
As a consequence, we have
\[
    H(z, \alpha, \beta) =-\beta^\top b -\frac{1}{4} \sum_{i=1}^k \alpha_i^2 + \sum_{j=1}^n (- \frac{\eta}{4} \gamma_j^2 + \theta) z_j,
\]
where $\gamma = c+ V \sqrt{\Lambda} \alpha+A^\top \beta$ and $\gamma_j$ is the $j$-th element of $\gamma$. Rewriting the summations using norm and matrix multiplications completes the proof.
\end{proof}

\begin{lemma}[Closed-form minimizer]\label{lem:closed form H}
Given any pair~$(\alpha, \beta)$, the minimizer of the function~$H$ defined in~\eqref{eq:H} can be computed as
\begin{align*}
    &\arg\min_{z \in \{0, 1\}^n} H(z, \alpha, \beta) =
    \mathbb{I}\big\{ \frac{\eta}{4} \diag((c\!+\!V \sqrt{\Lambda} \alpha\!+\!A^\top \beta) (c\!+\!V \sqrt{\Lambda} \alpha\!+\!A^\top \beta)^\top) > \theta \big\},
\end{align*}
where $\mathbb I$ is the component-wise indicator function and the $\diag$ operator here returns the vector of diagonal elements of the input matrix. 
\end{lemma} 
This lemma immediately follows from \eqref{eq:H}.

\end{document}